\documentclass[a4paper,11pt,DIV=11,% decrease borders by 2 (std 10 for 11pt)
abstract=on% return to old style centered abstract
]{scrartcl}
\usepackage{algorithm, booktabs}
\usepackage{mathtools}
\mathtoolsset{showonlyrefs}
\usepackage{amsfonts}
\usepackage{amsmath}
\usepackage{amssymb}
\usepackage{faktor}
\usepackage{amscd}
\usepackage{array}
\usepackage{amsthm}
\usepackage{subfigure}
\usepackage{mathdots}
\usepackage{mathrsfs}
\usepackage{algorithm}
\usepackage[noend]{algpseudocode}
\usepackage{graphicx}
\usepackage{color}
\usepackage{pgf}
\usepackage{scrpage2}
\usepackage{multirow}
\usepackage{listings}
\usepackage{tikz}
\usepackage{graphicx}
\usepackage[left=2.5cm, right=2.5cm,top=2.5cm,bottom=2.5cm]{geometry}
\usepackage{scalerel}
\usepackage{hyperref}

\newcommand{\N}{\ensuremath{\mathbb{N}}}

\newcommand{\R}{\ensuremath{\mathbb{R}}}

\newcommand{\tT}{\mathrm{T}}
\newcommand{\HH}{\mathcal{H}}
\newcommand{\HT}{{\mathcal{H}_{T}}}
\newcommand{\KK}{\mathcal{K}}

\newmuskip\pFqmuskip

\newcommand*\pFq[6][8]{
  \begingroup % only local assignments
  \pFqmuskip=#1mu\relax
  \begingroup\lccode`\~=`\,
  \lowercase{\endgroup\let~}\pFqcomma
  {}_{#2}F_{#3}{\left(\genfrac..{0pt}{}{#4}{#5};#6\right)}%
  \endgroup
}
\newcommand*\pRegFq[6][8]{
  \begingroup % only local assignments
  \pFqmuskip=#1mu\relax
  \begingroup\lccode`\~=`\,
  \lowercase{\endgroup\let~}\pFqcomma
  {}_{#2}\tilde{F}_{#3}{\left(\genfrac..{0pt}{}{#4}{#5};#6\right)}%
  \endgroup
}

\newcommand{\pFqcomma}{\mskip\pFqmuskip}

\DeclareMathOperator*{\argmin}{argmin}
\def\prox{\mathrm{prox}}
\def\Prox{\mathrm{Prox}}

\newtheorem{thm}{Theorem}[section]

\newtheorem{example}[thm]{Example}
\newtheorem{corollary}[thm]{Corollary}
\newtheorem{proposition}[thm]{Proposition}

\begin{document}
\title{Frame Soft Shrinkage as Proximity Operator 
}
%On the Interplay between Proximity and Linear Operators in Hilbert Spaces

\author{
Marzieh Hasannasab\footnotemark[1]
\and
Sebastian Neumayer\footnotemark[1]
	\and
	Gerlind Plonka\footnotemark[2]
	\and
	Simon Setzer\footnotemark[3]
	\and
	Gabriele Steidl\footnotemark[1]
	\and
	Jakob Geppert\footnotemark[2]
}

\maketitle

\footnotetext[1]{Department of Mathematics,
	Technische Universit\"at Kaiserslautern,
	Paul-Ehrlich-Str.~31, D-67663 Kaiserslautern, Germany,
	\{name\}@mathematik.uni-kl.de.}
	
	\footnotetext[2]{Institute for Numerical and Applied Mathematics,
	G\"ottingen University, Lotzestr.\ 16-18, 37083 G\"ottingen, Germany,
	\{name\}@math.uni-goettingen.de.
	}
	
	\footnotetext[3]{simon.setzer@gmail.com.}

\begin{abstract}
Let $\HH$ and $\KK$ be real Hilbert spaces 
and $T \in \mathcal{B} (\HH,\KK)$  an injective operator with closed range and Moore-Penrose inverse $T^\dagger$.
Based on the well-known characterization of proximity operators by Moreau, 
we prove that for any proximity operator $\Prox \colon \KK \to \KK$ the operator $T^\dagger \, \Prox \, T$ is a proximity operator on the linear space $\HH$ equipped with a suitable norm.
In particular, it follows for the frequently applied soft shrinkage operator $\Prox = S_{\lambda}\colon \ell_2 \rightarrow \ell_2$ and any frame analysis operator $T\colon \HH \to \ell_2$, 
that the frame shrinkage operator $T^\dagger\, S_\lambda\, T$ is a proximity operator in a suitable Hilbert space.
\end{abstract}

\section{Introduction}

Wavelet and frame shrinkage operators became very popular in recent years.
A certain starting point was the ISTA algorithm in \cite{DDM04} which was interpreted as a special case of
the forward-backward algorithm  in \cite{CW05}. 
For relations with other algorithms see also \cite{BSS2017,Se09}.
Let  $T \in  \mathbb R^{n \times d}$, $n\ge d$ have full column rank.
Then, the setting
\begin{equation} \label{eq:problem}
\argmin_{y \in \mathbb R^d} \Bigl\{ \frac12 \|x -y \|_2^2 + \lambda \|Tx\|_{1} \Bigr\}, \quad \lambda >0 
\end{equation}
is known as analysis point of view, while
the ,,opposite'' case  $T \in \mathbb R^{n \times d}$ with $n \le d$ can be seen as synthesis approach, see \cite{EMR07}.
For orthogonal $T \in {\mathbb R}^{d \times d}$, the frame soft shrinkage operator 
$T^\dagger \, S_\lambda \, T  = T^\tT \, S_\lambda \, T$ is the solution  of the above problem.
In the synthesis case, with the additional assumption $T T^\tT = I_n$, the solution of problem \eqref{eq:problem} is given by
$(I_d -  T^\tT T) x + T^\tT S_\lambda T$,
see \cite[Theorem 6.15]{Beck17}.
For arbitrary $T \in {\mathbb R}^{n \times d}$, $n \ge d$, 
there do not exist analytic expressions for the solution of \eqref{eq:problem}
in the literature.

Recently, some authors of this paper considered the question, 
whether the frame soft shrinkage operator can itself be seen as a proximity  operator, see  \cite{GP2019}.
They showed that the set-valued operator $(T^\dagger S_\lambda T)^{-1} - I_d$
is maximally cyclically monotone, which implies that it is a proximity operator with respect to some norm in
$\mathbb R^d$.
In this short paper, we prove that 
for any injective operator $T\in \mathcal{B} (\HH,\KK)$ with closed range and
any proximity operator $\Prox\colon \KK \to \KK$, 
the new operator $T^\dagger \, \Prox \, T\colon \HH \rightarrow \HH$ 
is also a proximity operator on the linear space $\HH$, 
but equipped with another inner product. This includes the above mentioned finite dimensional setting as a special case.
In contrast to \cite{GP2019}, we directly approach the question using a classical result of Moreau \cite{Moreau65}. 
Moreover, we explicitly provide the function for the definition of the proximity operator.

There are several fields where our results may be of interest.
So different norms in the definition of the proximity  operator were successfully used in variable metric algorithms, see \cite{CPR2013}.
Recently, it was shown that many activation functions appearing in neural networks are indeed proximity functions \cite{CP2018}.
Here we concatenate proximity operators with a linear operators and stay within the set of proximity operators just in different Hilbert spaces.
Finally, our findings may be of interest in connection with so-called Plug-and-Play algorithms \cite{CWE2016,SVW2016,TBF2018}.

Our paper is organized as follows: 
We start with preliminaries on convex analysis in Hilbert spaces in Section \ref{sec:prelim}.
In Section \ref{sec:interplay}, we prove our general results on the interplay between proximity and certain linear operators.
As a special case we emphasize that the frame soft shrinkage operator is itself a proximity operator.

%-------------------------------------------------------------------------------------
\section{Preliminaries} \label{sec:prelim}
Let $\HH$ be a real Hilbert space with inner product $\langle \cdot,\cdot \rangle$ and norm $\| \cdot \|$. 
By  $\Gamma_0(\HH)$ we denote the set of proper, convex, lower semi-continuous functions on $\HH$ mapping into $(-\infty,  \infty]$. %$\mathbb R$.
For $f \in \Gamma_0(\HH)$ and $\lambda > 0$, consider the proximity operator $\prox_{\lambda f}\colon \HH \rightarrow \HH$ and its Moreau envelope
$M_{\lambda f}\colon \HH \rightarrow \R$ defined by
\begin{align}
\prox_{\lambda f} (x) &\coloneqq \argmin_{y \in \HH} \Bigl\{ \frac12 \|x-y\|^2 + \lambda f(y) \Bigr\}, \\
M_{\lambda f} (x) &\coloneqq \min_{y \in \HH} \Bigl\{ \frac12 \|x-y\|^2 + \lambda f(y) \Bigr\}.
\end{align}
Clearly, the proximity operator and its Moreau envelope depend on the underlying space $\HH$, in particular on the chosen inner product.
Recall that an operator $G\colon \HH \to \HH$ is called firmly nonexpansive if for all $x,y \in \HH$ the following relation
is fulfilled
$$
\|Gx -Gy\|^2 \le \bigl\langle x-y,Gx-Gy \bigr\rangle .
$$
Obviously, firmly nonexpansive operators are nonexpansive.

For a Fr\'echet differentiable functions $\Phi\colon \HH \to \R$, 
the gradient $\nabla \Phi(x)$ at $x \in \HH$ is defined as the vector satisfying for all $h \in \HH$,
\[\langle \nabla \Phi(x), h \rangle = D\Phi(x) h,\]
where $D\Phi\colon \HH\to \mathcal{B} (\HH,\R)$ denotes the Fr\'echet derivative of $\Phi$,
i.e., for all $x,h \in H$,
\begin{equation} \label{frechet}
\Phi(x+h) - \Phi(x) = D\Phi(x) h + o(\|h\|).
\end{equation}
Note that the gradient crucially depends on the chosen inner product in $\HH$.
The following results can be found, e.g., in \cite[Props.~12.27, 12.29]{BC11}. %,\cite[Thms.~6.42 and 6.60]{Beck17}, \cite{Ro97}.

\begin{thm} \label{lem:1}
Let $f \in \Gamma_0(\HH)$. Then the following holds true:
\\
i) The operator $\prox_{\lambda f} \colon \HH \to \HH$ is firmly  nonexpansive.
\\
ii) The function $M_{\lambda f}$ is  (Fr\'echet) differentiable and has a
Lipschitz-continuous gradient given by
\[\nabla  M_{\lambda f}(x) = x - \prox_{\lambda f}(x).\]
\end{thm}

Clearly, ii) implies that
\begin{equation}\label{proxi}
\prox_{\lambda f} (x) = \nabla \left( \frac12 \| \cdot \|^2 - M_{\lambda f}(x)  \right) = \nabla \Phi(x),
\end{equation}
where $\Phi \coloneqq \frac12 \|x\|^2 - M_{\lambda f}$ is convex since $\prox_{\lambda f}$ is nonexpansive \cite[Prop.~17.10]{BC11}.
Further, it was shown by Moreau that also the %reverse 
following (reverse) statement holds true \cite[Cor.~10c]{Moreau65}.

\begin{thm} \label{thm:1}
 The operator $\Prox\colon  \HH \rightarrow \HH$ is a proximity operator 
 if and only if it is nonexpansive and there exists a function $\Psi \in \Gamma_0(\HH)$, 
 such that for any $x \in \HH$ we have $\Prox(x) \in \partial \Psi(x)$, where $\partial \Psi$ denotes the subdifferential of $\Psi$
\end{thm}

Thanks to \eqref{proxi}, we conclude that $\Prox\colon  \HH \rightarrow \HH$ is a proximity operator 
if and only if it is nonexpansive and the gradient of a convex, differentiable function $\Phi \colon \HH \to \R$.
Note that recently, the characterization of Bregman proximity operators in a more general setting was discussed in \cite{GN2018}.
In the following example, we recall the Moreau envelope and the proximity operator related to the soft thresholding operator.

\begin{example} \label{ex:1}
	Let $\HH = \R$ with usual norm $|x|$ and $f(x) \coloneqq |x|$.
	Then, $\prox_{\lambda f}$ is the soft thresholding operator $S_\lambda$ defined by
	\[S_\lambda(x)\coloneqq 	\left\{
	\begin{array}{cl}
	x - \lambda& \mathrm{for} \; x > \lambda,\\
	0         & \mathrm{for} \; x \in [-\lambda,\lambda],\\
	x + \lambda& \mathrm{for} \; x < -\lambda,
	\end{array}
	\right.\]
	and $m_{\lambda | \cdot|}$ is the Huber function
	\[
	m_{\lambda | \cdot|} (x)
	=
	\left\{
	\begin{array}{cl}
	x - \frac{\lambda}{2} & \mathrm{for} \; x > \lambda,\\[0.5ex]
	\frac{1}{2\lambda} x^2         & \mathrm{for} \; x \in [-\lambda,\lambda],\\[0.5ex]
	- x - \frac{\lambda}{2}& \mathrm{for} \; x < -\lambda.
	\end{array}
	\right.\]
	Hence, $\prox_{\lambda f} = \nabla \varphi$, where $\varphi(x)  = \frac{x^{2}}{2} - m_{\lambda | \cdot|}(x)$, i.e., 
	\[\varphi(x) = 
	\left\{\begin{array}{cl}
	\tfrac12(x-\lambda)^2& \mathrm{for} \; x > \lambda,\\
	0 &\mathrm{for} \; x \in [-\lambda,\lambda],\\[0.5ex]
	\tfrac12 (x + \lambda)^2& \mathrm{for} \; x < -\lambda.
	\end{array}
	\right.\]
	For $\HH = \mathbb R^d$ and $f(x) \coloneqq \|x\|_1$, we can just use a componentwise approach.
	Then $S_\lambda$ is defined componentwise, the Moreau envelope reads as
	$M_{\lambda \| \cdot \|_1} (x) = \sum_{i=1}^d m_{\lambda | \cdot|}(x_i)$ 
	and the potential of $\prox_{\lambda \| \cdot \|_1}$ is $\Phi(x) = \sum_{i=1}^d \varphi(x_i)$.
\end{example}

%-----------------------------------------
\section{On the interplay between proximity and linear operators} \label{sec:interplay}

Let $\HH$ and $\KK$ be real Hilbert spaces with inner products $\langle \cdot,\cdot \rangle_\HH$ and  $\langle \cdot,\cdot \rangle_\KK$ and corresponding norms
$\| \cdot \|_\HH$ and $\| \cdot \|_\KK$, respectively.
By $\mathcal{B}(\HH,\KK)$ we denote the space of bounded, linear operators from $\HH$ to $\KK$ with domain $\HH$.
In this section, we show that for any injective operator $T \in \mathcal{B}(\HH,\KK)$ with closed range $\mathcal{R}(T)$
and any proximity operator $\Prox \colon \KK \to \KK$, the operator $T^{\dagger} \, \Prox \, T\colon \HH \rightarrow \HH$ 
is itself a proximity operator on the linear space $\HH$ equipped with a suitable (equivalent) norm $\| \cdot \|_{\HT}$, i.e., there exits a function $f \in \Gamma_0(\HH)$ such that
\[
T^{\dagger} \, \Prox \, T (x) = \argmin_{y \in \HH} \Bigl\{ \frac12 \|x-y\|_\HT^2 + f(y) \Bigr\}.
\]
We prove our main result in the next subsection and consider the important case
of frame soft shrinkage subsequently.

\subsection{Main Result} \label{sec:main}
For any injective $T \in \mathcal{B}(\HH,\KK)$ with closed range, the Moore-Penrose inverse (generalized inverse, pseudo-inverse) $T^\dagger \in \mathcal{B}(\KK,\HH)$
is explicitly given by
$$T^{\dagger} = (T^*T)^{-1} T^*.$$
It satisfies
$$T^{\dagger} T = \text{Id}, \quad T T^{\dagger} = P_{\mathcal{R}(T)},\quad T^\dagger = T^\dagger P_{\mathcal{R}(T)},$$
where  $P_{\mathcal{R}(T)}$ is the orthogonal projection onto $\mathcal{R}(T)$,   see \cite[Ex.~3.27 \& Prop.~3.28]{BC11}.
%Note that $T^\dagger$ is indeed the synthesis operator for the canonical dual frame of $\{ f_k\}_{k\in\N}$.

Every injective $T\in \mathcal{B}(\HH,\KK)$ gives rise to an inner product in $\HH$ via $\langle x, y \rangle_\HT = \langle Tx, Ty \rangle_\HH$ 
and corresponding norm $\|x\|_\HT=\| Tx\|_\HH$. 
In general this defines a pre-Hilbert space.
Since $T$ has additionally closed range, the norms $\| \cdot\|_\HH$ and $\| \cdot\|_\HT$ are equivalent on $\HH$ due to
\[\Vert T\Vert_{\mathcal{B}(\HH,\KK)}^{-1} \| x\|_\HT \leq \| x\|_\HH = \Vert T^{\dagger} T x \Vert_\HH \leq \Vert T^{\dagger}\Vert_{\mathcal{B}(\KK,\HH)} \Vert x \Vert_\HT\]
for all $x \in \HH$. The norm equivalence also ensures the completeness of $\HH$ equipped with the new norm.
To emphasize that we consider the linear space $\HH$ with this norm we write $\HT$.

For a Fr\'echet differentiable function $\Phi \colon \HH \to \R$, the gradient $\nabla_\HT \Phi(x)$ at $x \in \HH$ in the space $\HT$ is given by the vector satisfying
\[\langle \nabla_\HT \Phi(x), h \rangle_\HT = D\Phi(x) h = \langle \nabla_\HH \Phi(x), h \rangle_\HH\]
for all $h \in \HH$.
Hence, the gradient depends on the chosen norm through
$$
\nabla_\HT \, \Phi(x) = (T^* T )^{-1}\nabla_\HH \Phi(x).
$$
Now, the desired result follows from the next theorem.

\begin{thm}\label{thm:Existence}
Let $T \in \mathcal{B}(\HH,\KK)$ be injective with closed range and $\Prox \colon \KK \to \KK$
a proximity operator on $\KK$.
Then, the operator $T^\dagger \, \Prox \, T \colon \HT \to \HT$ is a  proximity operator.
\end{thm}

\begin{proof}
In view of Theorems \ref{lem:1} and \ref{thm:1}, 
it suffices to show that $T^\dagger \, \Prox \, T$
is nonexpansive and that there exists a convex function $\Psi \colon \HT \to \R$ with  $T^\dagger \, \Prox \, T  = \nabla_\HT \Psi$.

1.~First, we show that $T^\dagger  \, \Prox \, T$ is firmly nonexpansive, and thus nonexpansive.
Using that  $T T^\dagger = P_{\mathcal{R}(T)}$, it follows
\begin{align}
\Vert T^\dagger  \, \Prox \,  T x - T^\dagger  \, \Prox \,  T y \Vert_\HT^2 
= \Vert  T T^\dagger \left(  \, \Prox \,  T x -  \Prox \,  T y \right)\Vert_\KK^2
\leq \Vert \Prox \, T x -   \Prox \,  T y \Vert_\KK^2. \label{xx}
\end{align}
Now, we obtain
\begin{align}
\bigl \langle  T^\dagger \, \Prox \, T x - T^\dagger \, \Prox \, T y, x-y \bigr \rangle_\HT
&= 
\bigl\langle T T^\dagger \bigl( \Prox \, T x - \Prox \, T y \bigr), Tx - Ty \bigr\rangle_\KK\\ 
&=  
\bigl\langle T^*T T^\dagger \bigl( \Prox \, T x - \Prox \, T y \bigr), x - y \bigr\rangle_\HH\\
&=
\bigl\langle \Prox \, T x - \Prox \, T y, Tx - Ty \bigr\rangle_\KK,
\end{align}
and since $\Prox$ is  firmly nonexpansive with respect to $\|\cdot\|_{\KK}$ and by \eqref{xx} further
\begin{align}
\bigl \langle  T^\dagger \, \Prox \, T x - T^\dagger \, \Prox \, T y, x-y \bigr \rangle_\HT
&\geq 
\Vert \Prox \, T x - \Prox \, T y \Vert^2_\KK\\
&
\geq 
\Vert T^\dagger \, \Prox \, T x - T^\dagger \, \Prox \, T y \Vert_\HT^2.
\end{align}
Thus, $T^{\dagger} \, \Prox \, T$ is firmly nonexpansive.

2.~It remains to prove that there exists a convex function $\Psi\colon \HT \to \R$ with
$\nabla_\HT \Psi = T^{\dagger} \, \Prox \, T$.
Since $\Prox$ is a proximity operator, there exists $\Phi \colon \HH \to\R$ with $\Prox = \nabla_\KK \Phi$.
Then, a natural candidate is given by $\Psi =\Phi \, T$.
Using the definition of the gradient and the chain rule, it holds for all $x,h\in \HH$ that
\begin{align}
\langle \nabla_\HH \Psi(x) , h \rangle_\HH =  
D\Psi (x)h = D\Phi(Tx)\,Th = 
\langle \nabla_\KK \Phi(Tx) , Th \rangle_\KK =   \langle T^* \, \Prox \, Tx , h \rangle_\HH.
\end{align}
Hence, we conclude $ \nabla_\HT \Psi  =  (T^* T )^{-1}\nabla_\HH \Psi(x) = (T^* T )^{-1} T^* \, \Prox \, Tx = T^\dag \, \Prox \, Tx$.
Finally, $\Psi$ is convex since it is the concatenation of a convex function with a linear function.
\end{proof}

For $\Prox \coloneqq \prox_{g}$ 
with $g \in \Gamma_0(\HH)$,
we are actually able to explicitly compute 
$f \in \Gamma_0(\HH)$ such that $T^\dagger \, \Prox \, T = \prox_f$ on $\HT$.
Clearly, this also gives an alternative proof for Theorem~\ref{thm:Existence}.

\begin{thm}\label{thm:ExpForm}
Let $T \in \mathcal{B}(\HH,\KK)$ be an injective operator with closed range 
and $\Prox \coloneqq \prox_{g}$ for some $g \in \Gamma_0(\KK)$.
Then $T^{\dagger} \, \prox_{g} \, T \colon \HT \to \HT$ is the proximity operator on $\HT$ of $f \in \Gamma_0(\HH)$ given by
\begin{equation} \label{prox_expl}
f(x) 
\coloneqq 
g  \square \Bigl( \tfrac12 \| \cdot \|_{\KK}^2 + \iota_{\mathcal{N}(T^*)}  \Bigr)  (Tx),
\end{equation}
where $\mathcal{N}(T^*)$ denotes the kernel of $T^*$.
For bijective $T \in \mathcal{B}(\HH,\KK)$ this expression simplifies to 
$
f(x) =  g(Tx).
$
\end{thm}

\begin{proof}
Since $\mathcal{R}(T)$ is closed, we have the orthogonal decomposition 
$\HH = \mathcal{R}(T) \oplus \mathcal{N}(T^*)$.
Then, we get
\begin{align}
T^\dagger \,  \prox_{g} \, (T x) 
&=
T^\dagger \argmin_{z \in \KK} \Bigl\{ \tfrac12 \|z - Tx\|_{\KK}^2 +  g(z)\Bigr\}\\
&=
T^\dagger P_{\mathcal{R}(T)} \argmin_{z = z_1+z_2, z_1 \in \mathcal{R}(T), z_2 \in \mathcal{N}(T^*)}
\Bigl\{ \frac12 \|z_1 + z_2 - Tx\|_{\KK}^2 +  g(z_1 + z_2) \Bigr\}\\
&=
T^\dagger \argmin_{z_1 \in \mathcal{R}(T)} \inf_{z_2 \in  \mathcal{N}(T^*)}  
\Bigl\{ \frac12 \|z_1 - Tx\|_{\KK}^2 + \frac12 \|z_2\|_{\KK}^2 +  g(z_1 + z_2)  \Bigr\}\\
&= 
T^\dagger \argmin_{z_1 \in \mathcal{R}(T)} \left\{ \frac12 \|z_1 - Tx\|_{\KK}^2 
+ \inf_{z_2 \in  \mathcal{N}(T^*)} \Bigl\{ \frac12\|z_2\|_{\KK}^2 + g(z_1 + z_2)  \Bigr\} \right\}\\
&=
T^\dagger T \argmin_{y \in \HH} \left\{ 
\frac12 \|Ty - Tx\|_{\KK}^2 
+ 
\inf_{z_2 \in  \mathcal{N}(T^*)} 
\Bigl\{ \frac12\|z_2\|_{\KK}^2 + g(Ty + z_2) \Bigr\}
\right\}\\
&=
\argmin_{y \in \HH} 
\left\{ \frac12 \|y - x\|_\HT^2 + \inf_{z_2 \in  \mathcal{N}(T^*)} 
\Bigl\{ \frac12\|z_2\|_{\KK}^2 
+  g(Ty + z_2)  \Bigr\} \right\} \label{magic}\\
&= 
\argmin_{y \in \HH} \left\{ \frac12 \|y - x\|_\HT^2 + 
 g \square \Bigl( \tfrac12 \| \cdot \|_{\KK}^2 + \iota_{\mathcal{N}(T^*)}  \Bigr) (Ty) \right\},
\end{align}
where $f\square g(x) \coloneqq \inf_{y\in \HH } f(y) + g(x-y)$ 
denotes the infimal convolution of $f,g \in \Gamma_0(\HH)$ and $x \mapsto \iota_S(x)$ is the indicator function of the set $S$ taking the value $0$ if $x \in S$ and $+\infty$ otherwise.
Hence, we conclude that $T^{\dagger} \, \prox_{g} \, T$ 
is the proximity operator on $\HT$ of $f$ in \eqref{prox_expl}.
In particular, we conclude for bijective $T$ by \eqref{magic} that 
\[
T^{\dag} \, \prox_{g}  \, (T x) 
= 
\argmin_{y \in \HH} \Bigl\{
\frac12 \| x-y\|_\HT^2 +   g(Ty) 
\Bigr\}.
\]
\end{proof}

Note that, in general $f$ is  a weaker regularizer than $g$, i.e.,  $f \leq  g$. 
This is necessary since for the latter we would get using the same reasoning as in \eqref{magic}
\[
\argmin_{y \in \HH} \Bigl\{\frac12 \| x-y\|_T^2 +   g(Ty) \Bigr\} 
= T^\dagger \argmin_{z \in \KK} 
\Bigl\{ \frac12 \|z - Tx\|_\KK^2 + g(z) +  \iota_{{\cal R}(T)}(z)\Bigr\} \neq T^\dagger \prox_{g} (T x).
\]

\subsection{Frame Soft Shrinkage}

In this section, we investigate the frame soft shrinkage as a special proximity operator.
Let $\KK = \ell_2$ be the Hilbert space of quadratic summable sequences $c = \{c_k\}_{k \in \mathbb N}$ 
with norm $\|c \|_{\ell_2} \coloneqq ( \sum_{k \in \mathbb N} |c_k|^2)^{\frac12}$ and assume further that $\HH$ is separable.
A set $\{x_k\}_{k\in\N}$, $x_k \in \HH$ is called a frame of $\HH$, if there exist constants $0 < A \le B < \infty$ such that
for all $x \in \HH$,
\begin{equation} \label{frame}
A \|x\|_\HH^2 \le \sum_{k\in \N} |\langle x,x_k \rangle_\HH |^2 \le B \|x\|_\HH^2.
\end{equation}
Given a frame $\{x_k\}_{k\in\N}$  of $\HH$, 
the  corresponding analysis operator  $T \colon \HH \to\ell_2$ is defined as 
$$Tx=\{ \langle x,x_k \rangle_\HH\}_{k\in\N}, \quad x\in \HH.$$
Its adjoint $T^*\colon\ell_2 \to \HH$ is the synthesis operator given by 
$$T^*\{c_k\}_{k\in\N} = \sum_{k\in\N}  c_k x_k, \quad \{c_k\}_{k\in\N} \in\ell_2.$$ 
By composing $T$ and $T^*$, we get the frame operator 
$$T^*Tf = \sum_{k\in\N} \langle x , x_k \rangle_\HH x_k, \quad x\in \HH. $$
The frame operator $T^*T$ is invertible on $\HH$, see \cite{CB2016}, such that
$$f = \sum_{k\in\N} \langle x , x_k \rangle (T^*T)^{-1} x_k, \quad x\in \HH. $$
The sequence $\{ (T^*T)^{-1}x_k\}_{k\in\N}$ is called the canonical dual frame of $\{ x_k\}_{k\in\N}$. 
Note that $T^\dagger$ is indeed the synthesis operator for the canonical dual frame of $\{ f_k\}_{k\in\N}$.
The relation between linear, bounded, injective operators of closed range and frame analysis operators reads as follows:

\begin{proposition} \label{prop:1} 
	An operator $T \in \mathcal{B}(\HH,\ell_2)$ is injective and has closed range if and only if it is the analysis operator of some frame of $\HH$. 
\end{proposition}
\begin{proof}
	If $T$ is the analysis operator for a frame $\{x_k\}_{k\in\N}$, 
	then  $T$ is bounded, injective and has closed range, see \cite{CB2016}.
	Conversely, assume that $T \in \mathcal{B}(\HH,\ell_2)$ is injective and that $\mathcal{R}(T)$ is closed.
	By the closed range theorem, it holds $\mathcal{R}(T^*) = H$.
	Let $\{\delta_k\}_{k\in\N}$ be the canonical basis of $\ell_2$ 
	and set 
	$\{x_k \}_{k\in\N}\coloneqq \{T^{*} \delta_k\}_{k\in\N}$. 
	Since $\sum_{k\in \N} |\langle x,x_k \rangle_\HH |^2 = \Vert Tx \Vert_{\ell_2}^2$, we conclude that $\{x_k \}_{k\in\N}$ is a frame of $\HH$ and $T$ is the corresponding analysis operator.
\end{proof}

The soft shrinkage operator $S_\lambda$ on $\ell_2$ (applied componentwise)  
is the proximity operator corresponding to the function $g \coloneqq \lambda \| \cdot \|_1$, $\lambda>0$.
Then we obtain as immediate consequence of Theorem~\ref{thm:ExpForm} 
the following corollary.

\begin{corollary}\label{cor:f1}
Let $T\colon \HH \rightarrow \ell_2$ be an analysis frame operator and $\Prox\colon \ell_2 \to \ell_2$ an arbitrary proximity operator.
Then $T^{\dagger} \, \Prox \, T$ is itself a proximity operator on $\HH$ equipped with the norm
$\| \cdot \|_\HT$. In particular, it holds for $\Prox \coloneqq S_\lambda$, $\lambda >0$  that
$$
T^{\dagger} \, S_\lambda \, T (x) = \argmin_{y \in \HH} \left\{ \|x-y\|_\HT^2 +  f(y)\right\}, \quad 
f(y) \coloneqq \lambda \|\cdot\|_1 \square \Bigl( \tfrac12 \| \cdot \|_{\ell_2}^2 + \iota_{\mathcal{N}(T^*)}  \Bigr) (Ty).
$$
\end{corollary}

Finally, let us have a look at the finite dimensional setting
with $\HH \coloneqq \R^d$, $\KK \coloneqq \R^d$, $n\ge d$.
Then we have for any $T \in \R^{n,d}$ with full rank $d$ and an arbitrary proximity operator on $\R^n$ that
\begin{equation} \label{rechnen}
T^{\dagger} \, \Prox \, T (x) = \argmin_{y \in \R^d} \left\{ \frac12 \|x-y\|_T^2 + f(y) \right\}, \quad 
f(y) \coloneqq \lambda \|\cdot\|_1 \square \Bigl( \tfrac12 \| \cdot \|_{\ell_2}^2 + \iota_{\mathcal{N}(T^\tT)}  \Bigr) (Ty).
\end{equation}

\begin{example} \label{ex:2}
We want to compute $f$ for the matrix $T\colon \R^{1} \to \R^{2}$ 
given by $T = \begin{pmatrix} 1 & 2\end{pmatrix}^\tT$ and and the soft shrinkage operator $S_1$ on $\R^2$.
Note that this example was also considered in \cite{{GP2019}}.
By \eqref{rechnen} and since  $x = (x_1, x_2)^\tT \in \mathcal{N}(T^*)$ 
if and only if $x_1 = -2 x_2$ we obtain
\begin{align}
f(y) 
& =  \Vert \cdot \Vert_1 \square \Bigl( \frac12 \| \cdot \|_2^2 + \iota_{\mathcal{N}(T^*)} (\cdot) \Bigr)  (Ty)\\
&= \min_{Ty = z+x}  \left\{ \|z\|_1 +  \tfrac12 \| x \|_{\ell_2}^2 + \iota_{\mathcal{N}(T^\tT)}(x) \right\} 
=  \min_{x}  \left\{ \|Ty-x\|_1 +  \tfrac12 \| x \|_{\ell_2}^2 + \iota_{\mathcal{N}(T^\tT)}(x) \right\}\\
& = \min_{x \in \R^2} \bigl\Vert 
(y , 2y )^\tT - ( x_1, x_2)^\tT \bigr\Vert_1 + \frac{1}{2} \Vert x \Vert_2^2 + \iota_{N(T^\tT)}(x)\\
& =  \min_{x_2 \in \R}
\vert y +2x_2 \vert + \vert2y-x_2 \vert + \frac{5}{2} x_2^2.
\end{align}
Consider the strictly convex function  $g_y(x_2) = \vert y +2x_2 \vert + \vert2y-x_2 \vert + \frac{5}{2} x_2^2$.
For $\vert y \vert \leq \frac{2}{5}$, it holds
\begin{equation}
0 \in \partial_{x_2}  g_y \left(-\frac{y}{2} \right) = [-2,2] - \mathrm{sgn}(y) -\frac{5}{2} y.
\end{equation}
Hence, by Fermat's theorem, the unique minimizer of $g_y(x_2)$  is given by $-\frac{y}{2}$.
Consequently, we have for $\vert y \vert \leq \frac{2}{5}$ that
\begin{equation}
f(y) = \frac{5}{2} \vert y \vert + \frac{5}{8} y^2.
\end{equation}
For $\vert y \vert > \frac{2}{5}$, the function $g_y$ is differentiable in $-\frac15 \mathrm{sgn}(y)$ and it holds
\begin{equation}
\partial_{x_2}  g_y(-\frac15 \mathrm{sgn}(y)) = 2\mathrm{sgn}(y) - \mathrm{sgn}(y) -\mathrm{sgn}(y) = 0.
\end{equation}
Therefore, for $\vert y \vert > \frac{2}{5}$, the minimizer of $g_y$ is $-\frac15 \mathrm{sgn}(y)$ and
\begin{equation}
f(y) = 3\vert y \vert - \frac{1}{10}.
\end{equation}
\end{example}

\subsection*{Acknowledgments} 
We like to thank J.-C. Pesquet for pointing us to \cite{Combettes2018}, we were not aware of when writing this paper.
It appears that Remark 3.10 iv) in \cite{Combettes2018} can be somehow reformulated towards our setting.
Funding by the German Research Foundation (DFG) within the project STE 571/13-1, the RTG 1932 and the RTG 2088 is gratefully acknowledged.

\bibliographystyle{abbrv}
\bibliography{references}
\end{document}